\documentclass[10pt]{article}
\usepackage{geometry}
\usepackage[utf8]{inputenc} 
\usepackage[T1]{fontenc}    
\usepackage{microtype}
\usepackage{url}            
\usepackage{booktabs} 
\usepackage{hyperref}
\usepackage{authblk}
\usepackage{xcolor}         
\usepackage{amsfonts}       
\usepackage{nicefrac}       
\usepackage{soul}

\makeatletter
\def\blfootnote{\xdef\@thefnmark{}\@footnotetext}
\makeatother

\usepackage{amsmath,amssymb,amsthm}
\usepackage{graphicx,subcaption}
\usepackage{mathdots}
\usepackage{multirow}
\usepackage{csquotes}
\usepackage[capitalize,noabbrev]{cleveref}
\usepackage{enumitem}
\graphicspath{ {./figure/} }

\usepackage[ruled,vlined]{algorithm2e}
\SetAlFnt{\small}

\numberwithin{equation}{section}

\newtheorem{theorem}{Theorem}[section]

\newtheorem{proposition}[theorem]{Proposition}
\theoremstyle{remark}
\newtheorem{remark}[theorem]{Remark}

\crefname{assumption}{Assumption}{Assumptions}

\newcommand\norm[1]{\left\Vert#1\right\Vert}
\newcommand\R{\mathbb{R}}
\newcommand\bx{\mathbf{x}}

\newcommand\cL{\mathcal{L}}

\DeclareMathOperator{\diag}{diag}

\title{A Nonoverlapping Domain Decomposition Method for Extreme Learning Machines: Elliptic Problems
\blfootnote{This work was supported in part by the National Research Foundation (NRF) of Korea grant funded by the Korea government (MSIT) (No. RS2023-00208914), and in part by Basic Science Research Program through NRF funded by the Ministry of Education (No. RS2023-00247199).}}
\author{Chang-Ock Lee$^{a}$,
  Youngkyu Lee$^{b}$, and Byungeun Ryoo$^{a}$
}
\affil{$^{a}$Department of Mathematical Sciences, KAIST, Daejeon 34141, Korea}
\affil{$^{b}$Division of Applied Mathematics, Brown University, Providence, RI 02912, USA}
\date{}

\begin{document}
\maketitle

\begin{abstract}
Extreme learning machine~(ELM) is a methodology for solving partial differential equations~(PDEs) using a single hidden layer feed-forward neural network.
It presets the weight/bias coefficients in the hidden layer with random values, which remain fixed throughout the computation, and uses a linear least squares method for training the
parameters of the output layer of the neural network.
It is known to be much faster than Physics informed neural networks.
However, classical ELM is still computationally expensive when a high level of representation is desired in the solution as this requires solving a large least squares system.
In this paper, we propose a nonoverlapping domain decomposition method (DDM) for ELMs that not only reduces the training time of ELMs, but is also suitable for parallel computation.
In numerical analysis, DDMs have been widely studied to reduce the time to obtain finite element solutions for elliptic PDEs through parallel computation.
Among these approaches, nonoverlapping DDMs are attracting the most attention.
Motivated by these methods, we introduce local neural networks, which are valid only at corresponding subdomains, and an auxiliary variable at the interface.
We construct a system on the variable and the parameters of local neural networks.
A Schur complement system on the interface can be derived by eliminating the parameters of the output layer.
The auxiliary variable is then directly obtained by solving the reduced system after which the parameters for each local neural network are solved in parallel.
A method for initializing the hidden layer parameters suitable for high approximation quality in large systems is also proposed.
Numerical results that verify the acceleration performance of the proposed method with respect to the number of subdomains are presented.

\end{abstract}


{\small \textbf{AMS subject classifications~(2020)}
35J25, 65N55, 68T07
}

\section{Introduction}
\label{Sec:Int}
Recently, deep neural networks have shown great performance in computer vision~\cite{he2016deep,krizhevsky2012imagenet,lecun2015deep} and natural language processing problems~\cite{radford2015unsupervised,ronneberger2015u,zhang2017beyond}.
Meanwhile, since neural networks can approximate any continuous function according to the Universal Approximation Theorem~\cite{barron1993universal,siegel2020approximation}, much research is being actively conducted to apply neural networks to various problems.
Among them, research on solving partial differential equations~(PDEs) based on neural networks is attracting attention, and the most popular method is Physics Informed Neural Networks~(PINNs)~\cite{raissi2019physics}.

There is a different type of method, so called extreme learning machines (ELMs) \cite{huang2011extreme, huang2006extreme} to solve the collocation system for the parameters of the output layer.
In this method, the weight/bias coefficients in all the hidden layers of the neural network are preset to random values and fixed throughout the computation.
The weight coefficients in the output layer of the neural network are trained by a least squares computation, not by back-propagation of gradients.
There is also a method that combines ELMs and PINNs, called Physics informed ELM (PIELM)~\cite{dwivedi2020physics}. Instead of solving the collocation system used in ELM, it optimizes the parameters of the output layer by minimizing the physics informed loss. PIELM is the same as ELM except that it contains the number of collocation points in the induced least squares system.

\begin{figure}
  \centering
  \includegraphics[width=0.5\linewidth, trim=7cm 6cm 7cm 6cm,clip]{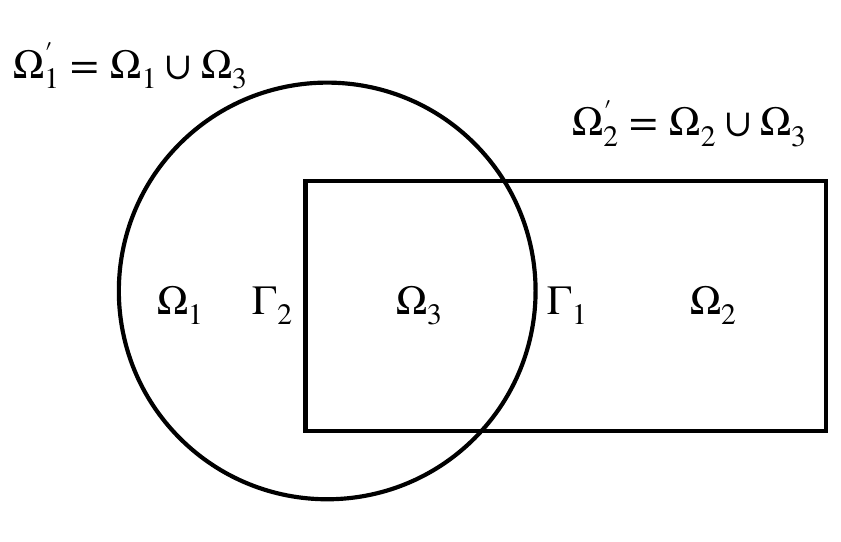}
  \caption{The domain $\Omega$ is the union of two overlapping subdomains $\Omega_{1}^{'}=\Omega_{1} \cup \Omega_{3}$ and $\Omega_{2}^{'}=\Omega_{2} \cup \Omega_{3}$. Note that $\Gamma_{1}$ and $\Gamma_{2}$ are interfaces.}
  \label{fig:schwarz}
\end{figure}

In the field of numerical analysis, there are many fast solvers for PDEs.
The domain decomposition method (DDM) is one of them, which is suitable for parallel computing of large-scale problems and has been successfully developed for elliptic PDEs~\cite{dolean2015introduction,toselli2005domain}.
Additionally, research on applying DDMs to convex optimizations~\cite{badea2012one,park2020additive} and deep learning problems~\cite{gunther2020layer,lee2022two,lee2024parareal} is actively underway.
Depending on how the domain is divided, there are two types of DDMs: overlapping and nonoverlapping methods.
An overlapping method divides the domain into a union of subdomains as shown in~\cref{fig:schwarz}. See \cite{lions1988schwarz} for more details.
On the other hand, in a nonoverlapping method, the domain is divided into nonoverlapping subdomains with interfaces.
See~\cref{fig:nonoverlap} for a graphical description.
Typical nonoverlapping methods are the substructuring methods~\cite{bramble1986construction,farhat1991method} that solve the reduced problem defined on the interface for the interface unknowns and then solve the interior unknowns through parallel computation for each subdomain.
When performing a proper nonoverlapping decomposition, the added computational and communication costs are very small because the interface part is sparse.
Therefore, the interface part requires sequential computation, but does not become a bottleneck in the overall parallel computation.

\begin{figure}
  \centering
  \includegraphics[width=0.8\linewidth, trim=2cm 8cm 2cm 8cm,clip]{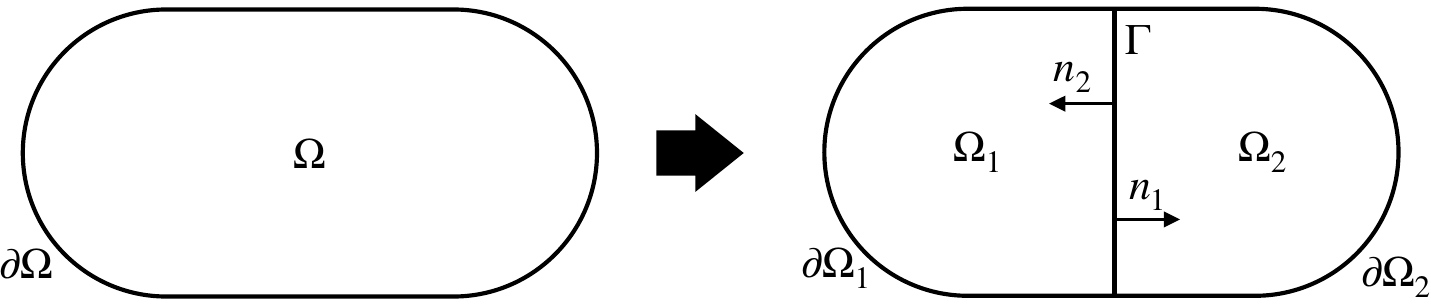}
  \caption{The domain $\Omega$ is decomposed into two nonoverlapping subdomains $\Omega_{1}$ and $\Omega_{2}$ with the interface $\Gamma$. Note that $n_{1}$ and $n_{2}$ are the outward normal vectors of $\Omega_{1}$ and $\Omega_{2}$, respectively, on $\Gamma$.}
  \label{fig:nonoverlap}
\end{figure}

Recently, several domain decomposition techniques have been developed to reduce the training time of PINNs.
DeepDDM~\cite{li2020deep} achieved training acceleration by directly applying the overlapping Schwarz method to PINN.
FBPINN~\cite{moseley2023finite} additionally designed a partition of unity function in the overlapping region to increase solution accuracy.
These methods have been extended to a two-level method using a coarse network~\cite{valentin2021coarse,yang2022additive}.
On the other hand, XPINN~\cite{jagtap2021extended} and cPINN~\cite{jagtap2020conservative} achieved training acceleration by adopting the idea of nonoverlapping domain decomposition in PINN.
However, since all of these methods require exchanging solution values or derivatives in each iteration, it is not easy to effectively reduce training time as the number of subdomains increases.
There are also domain decomposition type approaches for ELM and PIELM, so called Local ELM~(LocELM)~\cite{dong2021local} and Distributed PIELM (DPIELM)~\cite{dwivedi2020physics}, respectively, which are essentially the same.
In these methods, $C^k$ continuity conditions are imposed on the subdomain interfaces and the solution on each subdomain is produced by a local feed-forward neural network.
Each local neural network has small width in the hidden layers and parameters of the local neural network are assigned or optimized using the idea of ELM.
These methods showed that numerical errors typically decrease exponentially or nearly exponentially as the number of training parameters, or the number of training data points, or the number of subdomains increases.
However, they cannot be considered true DDMs because the parameters of the output layer are not optimized in parallel for each subdomain.

Motivated by substructuring methods of the nonoverlapping DDM, we propose a new nonoverlapping DDM for ELM.
In the proposed method, a local neural network is introduced for each subdomain, which is valid only in the subdomain,
and an auxiliary variable $u_{\Gamma}$ is also presented, which represents the interface values of the solution.
We construct a system for $u_{\Gamma}$ and the parameters of the output layers of local neural networks.
The system can be reduced to a Schur complement system by eliminating the parameters of the local neural networks.
The auxiliary variable $u_{\Gamma}$ is easy to compute because it solves the reduced system, and the parameters of the local neural networks can also be easily solved in the least squares sense.
Compared to LocELM, the proposed nonoverlapping DDM performs parallel computation after solving the reduced system.
This feature helps the proposed method effectively reduce the training time of ELM even as the number of subdomains increases.

It is known that the performance of ELMs depends greatly on the distribution used to sample the hidden layer parameters~\cite{dong2021local}.
A trial and error process to determine a suitable distribution would be expensive for large scale problems.
To address this issue, we propose an initialization scheme for one hidden layer networks that scales well with the number of neurons.

The rest of this paper is organized as follows.
In Section~\ref{Sec:Prelim}, we briefly introduce classical DDMs and ELM.
We present a new nonoverlapping DDM for ELM motivated by the substructuring method in Section~\ref{Sec:Ndd}.
In Section~\ref{Sec:Num}, a new initialization scheme is presented, followed by acceleration performance of the proposed nonoverlapping DDM for ELM applied to Poisson's equation and a plate bending problem.
We conclude this paper with remarks in Section~\ref{Sec:Conc}.

\section{Preliminaries}\label{Sec:Prelim}
\label{Sec:DDM}
We consider a general PDE with boundary conditions:
\begin{equation}
  \label{eqn:pde}
  \left\{\begin{aligned}
  \cL u = f \quad &\text{ in } \Omega,\\
  \mathcal{B}u = g \quad &\text{ on } \partial \Omega,
  \end{aligned}\right.
\end{equation}
where $\Omega \subset \R^{d}$ is a bounded domain, $\partial \Omega$ denotes the boundary of $\Omega$, $f, g$ are functions, $\cL$ is a differential operator defined for a function $u$,
and $\mathcal{B}$ represents boundary conditions, e.g. $\mathcal{B}u=u=g$ gives Dirichlet boundary conditions.
We assume that the problem~\eqref{eqn:pde} is well-posed so that the solution exists uniquely.

\subsection{Nonoverlapping domain decomposition methods}
DDM is known to be one of the successful fast numerical solvers for PDEs.
In particular, it can be implemented efficiently in parallel, saving time spent on solving large-scale problems.
There are two types of DDMs: overlapping and nonoverlapping methods, depending on how the domain is divided.

This paper focuses on nonoverlapping methods.
Assume that the differential operator $\cL$ is a second-order differential operator and that the domain $\Omega$ is divided into two nonoverlapping subdomains $\Omega_{1}$ and $\Omega_{2}$.
Then, under suitable assumptions on $f$ and boundaries of the subdomains, the solution of~\eqref{eqn:pde} is the same as the solution of the following system:
\begin{equation}
  \label{eqn:nonoverlap}
  \left\{\begin{aligned}
  \cL u_{1} &= f  & & \text{ in } \Omega_{1},\\
  u_{1} &= g  & & \text{ on } \partial \Omega_{1} \backslash \Gamma , \\
  \cL u_{2} &= f  & & \text{ in } \Omega_{2},\\
  u_{2} &= g  & & \text{ on } \partial \Omega_{2} \backslash \Gamma , \\
  u_{1} &= u_{2} & & \text{ on } \Gamma, \\
  \frac{\partial u_{1}}{\partial n_{1}}&=-\frac{\partial u_{2}}{\partial n_{2}}  & & \text{ on } \Gamma,
  \end{aligned}\right.
\end{equation}
where $\Gamma$ denotes the interface between $\Omega_{1}$ and $\Omega_{2}$.
In addition, $n_{1}$ and $n_{2}$ are the outward normal to $\Omega_{1}$ and $\Omega_{2}$, respectively.

Among several nonoverlapping DDMs, we first focus on substructuring methods to eliminate interior unknowns inside subdomains to create a reduced Schur complement system defined on the interface and solve it using an iterative method.
Specifically, let $Ku = F$ be the finite element discretization of~\eqref{eqn:pde}.
Then, we can rearrange the system into interior and interface parts:
\begin{align*}
  K= \begin{bmatrix}
    K_{II}^{1} & 0 &K_{I\Gamma}^{1} \\
    0 & K_{II}^{2} &K_{I\Gamma}^{2} \\
    K_{\Gamma I}^{1} & K_{\Gamma I}^{2} &K_{\Gamma\Gamma}
   \end{bmatrix},\quad
   u = \begin{bmatrix}
    u_{I}^{1} \\
    u_{I}^{2} \\
    u_{\Gamma}
    \end{bmatrix},\quad
    F = \begin{bmatrix}
      F_{I}^{1} \\
      F_{I}^{2} \\
      F_{\Gamma}
      \end{bmatrix},
\end{align*}
where $K_{\Gamma \Gamma} = K^{1}_{\Gamma \Gamma} + K^{2}_{\Gamma \Gamma}$ and $F_{\Gamma}=F^{1}_{\Gamma}+F^{2}_{\Gamma}$.
Using $u_{I}^{s} = K^{s^{-1}}_{I I}(F^{s}_{I} - K^{s}_{I \Gamma}u_{\Gamma})$ for $s=1,2$, we can eliminate the interior unknowns and get the reduced system
\begin{align}
  \label{eqn:subs}
  (S^{1}+S^{2})u_{\Gamma}=G^{1}_{\Gamma} + G^{2}_{\Gamma},
\end{align}
where $S^{s} = K^{s}_{\Gamma\Gamma}-K^{s}_{\Gamma I}K^{s^{-1}}_{I I}K^{s}_{I \Gamma}$ and $G^{s}_{\Gamma} = F^{s}_{\Gamma} - K^{s}_{\Gamma I}K^{s^{-1}}_{I I} F^{s}_{I}$.
The reduced system~\eqref{eqn:subs} can be easily solved by the conjugate gradient method and can even be accelerated by introducing a preconditioner~\cite{dohrmann2003preconditioner,mandel1993balancing}.
Additionally, after finding $u_{\Gamma}$, we can do one-time communication and then find the interior unknowns through parallel computation.

\subsection{Universal approximation theorem for extreme learning machines}
ELM was developed for single hidden layer feed-forward neural networks \cite{huang2011extreme, huang2006extreme}.
It presets the weight/bias coefficients in all the hidden layers of the network with random values, which remain fixed throughout the computation,
and uses a linear least squares method for training the weight coefficients in the output
layer of the neural network.
ELM is one example of the so-called randomized neural networks (see e.g.~\cite{igelnik1995stochastic, pao1994learning}).
The applications of ELM to function approximations and linear differential equations have been considered in
several recent works~\cite{dwivedi2020physics, liu2020legendre, panghal2021optimization} based on the following Universal Approximation Theorem.
\begin{theorem}[\cite{huang2006universal, igelnik1995stochastic}]
Given any bounded non-constant piecewise
continuous function $g\colon \mathbb{R} \rightarrow \mathbb{R}$, for any continuous target function $f$, there exists a sequence of single-hidden-layer feed-forward neural networks with $g$ as the
activation function, with its hidden layer coefficients randomly generated, and
with its output layer coefficients appropriately chosen, such that $\lim_{n\rightarrow\infty} \|f-f_n\| = 0$, where $n$ is the number of hidden-layer nodes and $f_n$ is the output of
the neural network.
\end{theorem}

\section{Schur complement system for ELM}
\label{Sec:Ndd}
In this section, we present a new nonoverlapping DDM inspired by Schur complement systems in classical DDMs.

Let $u_{\theta}$ be a neural network solution for~\eqref{eqn:pde}.
Then, the solution can be represented by
\begin{equation*}
  \displaystyle u_{\theta}(\bx) = \sum_{j=1}^{n} c_{j} \phi_{j}(\bx), \quad \forall \bx \in \Omega,
\end{equation*}
where $n$ is the number of neurons in the last layer and $\{ \phi_{j}(\bx) \}$ are outputs of the hidden layer at a given $\bx$ with parameter $\theta$.
Let $N$ be the number of subdomains, and let the domain $\Omega$ be decomposed into
\begin{equation*}
  \displaystyle \Omega = \cup_{s=1}^{N} \Omega_{s},
\end{equation*}
where $\{\Omega_{s}\}$ are nonoverlapping subdomains.
For each subdomain $\Omega_{s}$, we define a local neural network to produce $u_{s}(\bx)$ as
\begin{equation*}
  \displaystyle u_{s}(\bx)=\sum_{j=1}^{n_{N}}c^{s}_{j} \phi^{s}_{j}(\bx), \quad \forall \bx \in \Omega_{s},
\end{equation*}
where $n_{N}:= n / N$, $\phi^{s}_{j}(\bx)$ is the$j$-th output of the hidden layer at a  given $\bx$, and $\theta^{s}$ is the parameter for $\{\phi^{s}_{j}\}$.

For simplicity, we assume $N=2$ and $\cL$ is a second-order differential operator.
Then the solution of~\eqref{eqn:pde} is the same as the solution of~\eqref{eqn:nonoverlap}.

Let $\{\bx^{s}_{i}, \bx^{s}_{\Gamma}, \bx^{s}_{b}\}$ be the interior, interface, and boundary points, respectively:
\begin{equation*}
       \bx^{s}_{i}=\begin{bmatrix}
              x^{s}_{i,1} \\
              x^{s}_{i,2} \\
              \vdots \\
              x^{s}_{i, m^{s}_{i}}
       \end{bmatrix},\quad
       \bx^{s}_{\Gamma}=\begin{bmatrix}
              x^{s}_{\Gamma,1} \\
              x^{s}_{\Gamma,2} \\
              \vdots \\
              x^{s}_{\Gamma, m^{s}_{\Gamma}}
       \end{bmatrix},\quad
       \bx^{s}_{b}=\begin{bmatrix}
              x^{s}_{b,1} \\
              x^{s}_{b,2} \\
              \vdots \\
              x^{s}_{b, m^{s}_{b}}
       \end{bmatrix},
\end{equation*}
where $\bx^{s}_{i} \in \R^{m^{s}_{i} \times d}, \bx^{s}_{\Gamma} \in \R^{m^{s}_{\Gamma} \times d}$, and $\bx^{s}_{b} \in \R^{m^{s}_{b} \times d}$. Note that $m_{i}^{s}, m_{\Gamma}^{s}$, and $m_{b}^{s}$ denote the number of interior, interface, and boundary points, respectively.
Also, we introduce an auxiliary variable $u_{\Gamma}$ satisfying
\begin{equation*}
  u_{1} = u_{2} = u_{\Gamma} \text{ on } \Gamma.
\end{equation*}
Then, we discretize \eqref{eqn:nonoverlap} as follows:
\begin{equation}
  \label{eqn:nonoverlap_discrete}
  \begin{bmatrix}
    K^{1} & 0 & B^{1} \\
    0 & K^{2} & B^{2} \\
    A^{1} & A^{2} & 0
\end{bmatrix}
\begin{bmatrix}
    c^{1} \\
    c^{2} \\
    u_{\Gamma}
\end{bmatrix}
=
\begin{bmatrix}
    F^{1} \\
    F^{2} \\
    0
\end{bmatrix},
\end{equation}
where
\begin{align*}
  K^{s} &= \begin{bmatrix}
    \cL \phi^{s}_{1}(\mathbf{x}^{s}_{i}) & \cdots & \cL \phi^{s}_{n_{N}}(\mathbf{x}^{s}_{i}) \\
    \phi^{s}_{1}(\mathbf{x}^{s}_{b}) & \cdots & \phi^{s}_{n_{N}}(\mathbf{x}^{s}_{b}) \\
    \phi^{s}_{1}(\mathbf{x}^{s}_{\Gamma}) & \cdots & \phi^{s}_{n_{N}}(\mathbf{x}^{s}_{\Gamma})
  \end{bmatrix},
  \quad
  A^{s} = \begin{bmatrix}
    \frac{\partial \phi^{s}_{1}}{\partial n_{s}}(\mathbf{x}^{s}_{\Gamma})
    & \cdots
    & \frac{\partial \phi^{s}_{n_{N}}}{\partial n_{s}}(\mathbf{x}^{s}_{\Gamma})
  \end{bmatrix},
  \\
  c^{s} &= \begin{bmatrix}
    c^{s}_{1} \\
    \vdots \\
    c^{s}_{n_{N}}
  \end{bmatrix},
  \quad
  B^{s} = \begin{bmatrix}
    0 \\
    0 \\
    -I
  \end{bmatrix}, \quad
  F^{s} = \begin{bmatrix}
    f(\mathbf{x}^{s}_{i}) \\
    g(\mathbf{x}^{s}_{b}) \\
    0
  \end{bmatrix},\quad \text{for } s=1,2.
\end{align*}
 
\begin{remark}
    When there are more than two subdomains,
    $\frac{\partial \phi^{s}_{j}}{\partial n_{s}}(x_{\Gamma, j}^s)$ can become ambiguous when $x_{\Gamma, j}^s$ is a common corner point of the subdomains.
    In this case, we add a row to $A^s$ for each edge to which $x_{\Gamma, j}^s$ belongs, where $n_s$ is the normal direction to that edge.
  \end{remark}

Typically, the total number of training points is larger than that of neurons in the last layer, i.e.,
\begin{equation*}
  \sum_{s=1}^{N} \left(m^{s}_{i}+m^{s}_{b}+m^{s}_{\Gamma}\right) > n
\end{equation*}
and we solve~\eqref{eqn:nonoverlap_discrete} by the least squares method.
In the following proposition, if the local neural network is a 1-hidden layer network, we can show that the least squares system for~\eqref{eqn:nonoverlap_discrete} has a unique solution with probability 1.

\begin{proposition}
  Assume that, in each subdomain $\Omega_s$, the interior, boundary, and interface points $\{\bx^{s}_{i}, \bx^{s}_{b}, \bx^{s}_{\Gamma}\}$ are distinct samples such that $m_i^s \ge n_N$.
  In addition, let $u_s$ be the output of the local neural network with a 1-hidden layer, i.e.,
  \begin{equation*}
    u_{s}(\bx) = \sum_{j=1}^{n_{N}}c^{s}_{j} \sigma (w_{j}^{s} \cdot \bx + b_{j}^{s}), \quad \forall \bx \in \Omega_s,
  \end{equation*}
  where $\sigma$ is a $C^{\infty}$ activation function of which all derivatives are independent of each  other and $\{ w_{j}^{s} \}$ and $\{ b_{j}^{s} \}$ are randomly chosen from continuous probability distributions on $\R^{d}$ and $\R$, respectively.
  Then the matrix in~\eqref{eqn:nonoverlap_discrete} has full column rank with probability 1 so that the least squares system for~\eqref{eqn:nonoverlap_discrete} has a unique solution with probability 1.

\end{proposition}
\begin{proof}
  Following the proof of~\cite[Theorem 2.1]{huang2006extreme}, we can easily prove that the set
  $\{ \cL \phi_{j}^{s}(\bx_{i}^{s}) \}_{j=1}^{n_N}$ is a linearly independent set with probability 1.
  This means that $K^{s}$ is a full column rank matrix with probability 1.
  Then, the block diagonal matrix $K:=\diag(K^{1}, \cdots, K^{N})$ also has full column rank with probability 1. Furthermore, the matrix $B:=\left[ B^{1}, \cdots, B^{N} \right]^{T}$ is obviously a full column rank matrix. Hence, the matrix in~\eqref{eqn:nonoverlap_discrete} has full column rank with probability 1 so that the least squares system for~\eqref{eqn:nonoverlap_discrete} has a unique solution with probability 1.

\end{proof}

The system~\eqref{eqn:nonoverlap_discrete} can be written as
\begin{equation}\label{eqn:discrete-simple}
\begin{bmatrix}
    K & B \\
    A  & 0
\end{bmatrix}
\begin{bmatrix}
    c \\
    u_{\Gamma}
\end{bmatrix}
=
\begin{bmatrix}
    F \\
    0
\end{bmatrix}
\end{equation}
where
\begin{align*}
K=\begin{bmatrix}
    K^1 & 0 \\
    0  & K^2
\end{bmatrix},\quad
A=\begin{bmatrix}
    A^1 & A^2
\end{bmatrix}, \quad
B=\begin{bmatrix}
    B^1 \\
    B^2
\end{bmatrix}, \quad
F=\begin{bmatrix}
    F^1 \\
    F^2
\end{bmatrix}, \quad\text{and }
c=\begin{bmatrix}
    c^1 \\
    c^2
\end{bmatrix}.
\end{align*}
Multiplying
\begin{align*}
\begin{bmatrix}
    I &  0 \\
    -AK^+  & I
\end{bmatrix}\quad (K^+=(K^TK)^{-1}K^T)
\end{align*}
to \eqref{eqn:discrete-simple}, we obtain
\begin{equation}\label{eqn:eliminated}
\begin{bmatrix}
    K & B \\
    0  & -AK^+B
\end{bmatrix}
\begin{bmatrix}
    c \\
    u_{\Gamma}
\end{bmatrix}
=
\begin{bmatrix}
    F \\
    -AK^+F
\end{bmatrix}.
\end{equation}
Instead of solving the least squares system of~\eqref{eqn:nonoverlap_discrete}, we solve the least squares system of~\eqref{eqn:eliminated}
\begin{align*}
\begin{bmatrix}
    K^T &  0 \\
    B^T  & -B^TK^{+T}A^T
\end{bmatrix}
\begin{bmatrix}
    K & B \\
    0  & -AK^+B
\end{bmatrix}
\begin{bmatrix}
    c \\
    u_{\Gamma}
\end{bmatrix}
=
\begin{bmatrix}
    K^T &  0 \\
    B^T  & -B^TK^{+T}A^T
\end{bmatrix}
\begin{bmatrix}
    F \\
    -AK^+F
\end{bmatrix}.
\end{align*}
That is, $(c, u_\Gamma)$ is the solution of
\begin{equation}\label{LS_system}
\left\{
\begin{aligned}
K^TKc+K^TBu_\Gamma & = K^TF, \\
B^TKc+B^T(I+K^{+T}A^TAK^+)Bu_\Gamma & =B^T(I+K^{+T}A^TAK^+)F.
\end{aligned}
\right.
\end{equation}
So, by eliminating $c$ from~\eqref{LS_system}, we obtain the Schur complement system
\begin{equation}~\label{eqn:schur}
B^T(I+K^{+T}A^TAK^+ -KK^+)Bu_\Gamma = B^T(I+K^{+T}A^TAK^+ -KK^+)F.
\end{equation}

Since the system~\eqref{eqn:schur} is SPD, we use the conjugate gradient method~(CGM)~\cite{hestenes1952methods} to solve the system.
After finding $u_{\Gamma}$, in each subdomain $\Omega_{s}$, we optimize the parameters of the output layer of each local neural network by solving $K^{s^{T}}K^{s}c^{s}=K^{s^{T}}(F^{s}-B^{s}u_\Gamma)$.
The whole process is described in Algorithm~\ref{alg:iterative}.

\begin{algorithm}
  Set $\theta^{s}$ randomly. \\
                \For{$s=1, \cdots, N$ in parallel}{
                       Construct $K^{s}, A^{s}, B^{s}$, and $F^{s}$
                }
                Solve $B^T(I+K^{+T}A^TAK^+ -KK^+)Bu_\Gamma = B^T(I+K^{+T}A^TAK^+ -KK^+)F$ \\
                
                \For{$s=1, \cdots, N$ in parallel}{
                  Solve $K^{s^{T}}K^{s}c^{s}=K^{s^{T}}(F^{s}-B^{s}u_\Gamma)$
                
                }
  \caption{Nonoverlapping DDM for ELM}
  \label{alg:iterative}
\end{algorithm}

\section{Numerical experiments}
\label{Sec:Num}
In this section, we present numerical results of the proposed nonoverlapping DDM for ELM.
We conduct experiments on various PDEs to demonstrate the performance of the proposed method.

The neural network used for ELM is a 1-hidden layer network with $n=2^{16}=65{,}536$ neurons.
The number of neurons in each local neural network is $n / N$ where $N$ is the number of subdomains.
The $\tanh$ function is used for the activation function and the total number of training points for ELM is $m=(2^{8}+ 1)^{2}=66{,}049$, which are distributed to each subdomain.
CGM is used to solve the reduced system for $u_{\Gamma}$, stopping when the relative residual is smaller than $10^{-9}$.

All algorithms are implemented using Pytorch~\cite{paszke2019pytorch} with mpi4py~\cite{dalcin2021mpi}.
Except for the $16\times16$ case, all experiments were conducted on a single node equipped with two Intel Xeon Gold 6448H (2.4GHz 32c) processors, 256GB RAM, and each process was allocated $64 / N$ cores.
For the case of $16\times16$, two nodes with two Intel Xeon Gold 6448H and three nodes with two Intel Xeon Gold 6348 (2.6GHz 28c) with 256GB RAM each were used, and each process was allocated one core.
The cluster is connected by a 100Gbps Infiniband network and runs on the CentOS 7 operating system.

\begin{remark}
  There is a reason why we used CPUs, not GPUs.
  Solving the least squares system for ELM with a lot of neurons and training points on GPUs leads to out-of-memory issues.
\end{remark}

\subsection{Initialization of parameters of the hidden layers}
\label{Sec:interface}
In Algorithm~\ref{alg:iterative}, we solve the system~\eqref{eqn:schur} using CGM which requires matrix-vector multiplications involving $B^T(I+K^{+T}A^TAK^+ -KK^+)B$.
When a 1-hidden layer network is applied to the local neural network, the output of the $j$-th neuron for a given $\bx$ can be described as
\begin{equation*}
  \phi^{s}_{j}(\bx)=\sigma \left( w_{j}^{s}\cdot \bx + b_{j}^{s} \right) \text{ where } w^{s}_{j} \in \R^{d} \text{ and } b^{s}_{j} \in \R.
\end{equation*}
To satisfy the basic concept of ELMs, the initialization of $\{ w^{s}_{j} \}$ and $\{ b^{s}_{j} \}$ must meet two conditions.
First, the span of $\{\phi^{s}_{j}\}_{j=1}^{n_N}$ should cover the solution space on each subdomain.
Second, $\{\phi^{s}_{j}\}_{j=1}^{n_N}$ should be sufficiently independent.
A certain way to achieve the first condition would be that for given $w\in\R^d$, $b\in\R$ and $\epsilon>0$, the probability that there is some $\phi^s_j$ for which
$\norm{w - w_j^s} < \epsilon$ and $\lvert b - b_j^s\rvert < \epsilon$ tends to $1$ as the number of neurons increases.
However, in such a situation, many outputs of neurons become redundant, violating the second condition.
As a compromise, we require that the probability merely be a number close to $1$ for some fixed $\epsilon$.

Let $\bar{B}_r(\Omega_s) = \left\{\xi\mid \mathrm{dist}(\xi, \Omega_s) \leq r(n_N)\right\}$.
Suppose that for some functions $l, r \colon \mathbb{N}\rightarrow \mathbb{R_+}$ we choose $w^{s}_{j}$ and $b^{s}_{j}$ in the following way:
\begin{equation*}
\label{eqn:init_be}
\begin{split}
  w^{s}_{j} &\sim U([-l(n_N), l(n_N)]^d), \\
  b^{s}_{j} &= -w^{s}_{j} \cdot \xi^{s}_{j} \quad \text{ for } \xi^{s}_{j} \sim U(\bar{B}_{r(n_N)}(\Omega_s)),
\end{split}
\end{equation*}
where $U$ denotes the uniform distribution.
It is easy to see that the first condition above requires
\begin{equation*}
  l(n_N), r(n_N) \rightarrow \infty \text{ as } n_N \rightarrow \infty.
\end{equation*}

In practice, however, since the bias $b^{s}_{j}$ determines the geometric position of $\phi_{j}^{s}$,
the neuron outputs corresponding to $\xi^s_j$ far from $\Omega_s$ are mostly linear in the subdomain for activation functions such as $\tanh$.
This causes $K^s$ to become ill-conditioned as $r(n_N)$ increases.
Based on experiments using $\tanh$ as activation, $r \equiv \mathrm{diam}(\Omega_s) / 2$ was found to be sufficient for the Poisson equation and the plate bending problem.
We now continue the discussion assuming $r$ is fixed.

For given $\hat{w}\neq 0$ and $\hat{b}$, let
\begin{align*}
  V_\epsilon(\hat{w}, \hat{b}) &= \left\{(w, \xi) \,\middle|\, \norm{w - \hat{w}} < \epsilon,\, \lvert \hat{b} - w\cdot \xi\rvert < \epsilon,\, \mathrm{dist}(\xi, \Omega_s) < r\right\},\\
  W_\epsilon(\hat{w}, \hat{b}) &= \left\{(w, \xi) \,\middle|\, \norm{w - \hat{w}} < \epsilon,\, \lvert \hat{b} - \hat{w}\cdot \xi\rvert < \epsilon,\, \mathrm{dist}(\xi, \Omega_s) < r\right\}.
\end{align*}
The probability that $(w_j^s, b_j^s)$ satisfies $\norm{\hat{w} - w_j^s} < \epsilon$ and $\lvert \hat{b} - b_j^s\rvert < \epsilon$ is the volume of $V_\epsilon(\hat{w}, \hat{b})$ divided by the volume of the set of all possible $(w, \xi)$.
In order to bound the volume of $V_\epsilon$, the volume of $W_\epsilon(\hat{w}, \hat{b})$ is employed as follows:
\begin{equation*}
  W_{\epsilon/(1+R)}(\hat{w}, \hat{b}) \subseteq V_\epsilon(\hat{w}, \hat{b}) \subseteq  W_{(1+R)\epsilon}(\hat{w}, \hat{b}),
\end{equation*}
where $R$ bounds the 2-norm of $\bar{B}_r(\Omega_s)$.
Since $W_\epsilon(\hat{w}, \hat{b})$ is the Cartesian product of an $\epsilon$-ball about $\hat{w}$ and
the intersection of $\left\{\xi \,\middle|\, \norm{\hat{b} - \hat{w}\cdot \xi} < \epsilon\right\}$, an $\epsilon/\norm{\hat{w}}$-neighborhood of a hyperplane, with $\bar{B}_r(\Omega_s)$,
its volume is proportional to $\epsilon^d \cdot \mathrm{diam}(\bar{B}_r(\Omega_s))^{d-1}\epsilon / \norm{\hat{w}}$.

The probability that there exists $j$ such that $\norm{\hat{w} - w_j^s} < \epsilon$ and $\lvert \hat{b} - b_j^s\rvert < \epsilon$ is
\begin{equation}
  \label{eqn:prob}
  1 - \left(1 - \frac{\mathrm{vol}(V_\epsilon(\hat{w}, \hat{b}))}{\mathrm{vol}([-l(n_N), l(n_N)]^d\times \bar{B}_r(\Omega_s))}\right)^{n_N}.
\end{equation}
As $n_N\rightarrow \infty$, the liminf of~\eqref{eqn:prob} is non-zero if and only if
\begin{equation*}
  \frac{\mathrm{vol}([-l(n_N), l(n_N)]^d\times \bar{B}_r(\Omega_s))}{\mathrm{vol}(V_\epsilon(\hat{w}, \hat{b}))} = O(n_N).
\end{equation*}
Since
\begin{equation*}
  \frac{\mathrm{vol}([-l(n_N), l(n_N)]^d\times \bar{B}_r(\Omega_s))}{\mathrm{vol}(V_\epsilon(\hat{w}, \hat{b}))}
  \approx \frac{l(n_N)^d \cdot \mathrm{diam}(\bar{B}_r(\Omega_s))^d}{\epsilon^d \cdot \mathrm{diam}(\bar{B}_r(\Omega_s))^{d-1}\epsilon / \norm{\hat{w}}},
\end{equation*}
with the proportionality constant depending on $d$, $R$, and the shape of $\Omega_s$, we have
\begin{equation}
  \label{eqn:l}
  l(n_N) = O\left(\frac{\epsilon^{1+1/d}}{\norm{\hat{w}}^{1/d}\mathrm{diam}(\bar{B}_r(\Omega_s))^{1/d}}n_N^{1/d}\right).
\end{equation}

Because $\norm{\hat{w}}^{1/d}$ is in the denominator, the suggested initialization
will always cause the limsup of the probability in~\eqref{eqn:prob} to tend to zero as $\norm{\hat{w}}$ becomes large,
unless $l(n_N) = o(n_N^{1/d})$, which would violate the second condition.
However, we can fix a bound on $\norm{\hat{w}}$ based on the steepness of the solution, where large $\norm{\hat{w}}$ corresponds to steep slopes in the solution.
Since $r$ is fixed and~\eqref{eqn:l} is not derived for $N \to \infty$, the diameter of $\bar{B}_r(\Omega_s)$ is absorbed into the proportionality constant, yielding $l = O(n_N^{1/d})$.

\begin{remark}
  We may remove the dependency of $l(n_N)$ on $\norm{\hat{w}}$ if $w_j^s$ is sampled using the probability density function $f(w)\approx \norm{w}\chi_{[-l(n_N),l(n_N)]^d}(w)$, where $\chi$ is the characteristic function.
  In this case, $l(n_N) = O(n_N^{1/d+1})$.
  We did not observe significant performance differences between this modified distribution and the proposed initialization method~\eqref{eqn:l} in experiments on the 2D Poisson equation with exact solutions $\sin(16\pi x)\sin(16\pi y)$ and $\sin(32\pi x)\sin(32\pi y)$.
  Nevertheless, we expect solutions requiring even larger weights to benefit from the modified distribution.
\end{remark}

To determine the proportionality constant, we first scale the subdomain $\Omega_s$ to the unit square.
Combining $n_{N} = n/N$ with the fact that scaling performs a $N^{1/d}$ multiple, the initialization for 2D problems is given by
\begin{equation}
  \label{eqn:init}
  \begin{split}
    w^{s}_{j} &\sim U([-l, l]^2)\quad \text{ for } l = c\sqrt{n}, \\
    b^{s}_{j} &= -w^{s}_{j} \cdot \xi^{s}_{j} \quad \text{ for } \xi^{s}_{j} \sim U(\bar{B}_r(\Omega_{s})).
  \end{split}
\end{equation}
The constant $c$ is obtained empirically; we used $c=1/8$
for the Poisson equation and $c=1/16$ for the plate bending problem, respectively.
We set $r= \mathrm{diam}(\Omega_s)/2$ for both problems.
Note that suitable $c$ and $r$ may be determined through experiments on smaller problems to reduce costs, and then applied to larger ones.

The widely used He initialization~\cite{he2015delving} and Glorot initialization~\cite{glorot2010understanding} adjust $w^{s}_{j}$ by $1/\sqrt{n}$ and set $b^{s}_{j}=0$ to match the variance of the inputs and outputs of each layer.
On the other hand, our initialization described in~\eqref{eqn:init} sets the bias according to the geometric information of each subdomain.
Applying the He or Glorot initialization to $w^{s}_{j}$ results in a smaller sampling range of $w^{s}_{j}$ as $n$ increases.
As a result, all basis functions $\{\phi^{s}_{j}\}$ have similar shapes at similar locations, which makes $K^{s}$ ill-conditioned and reduces their approximation quality.

\begin{table}
  \caption{Square root of the mean square error between computed $u_{\Gamma}$ and the exact value of 2-D Poisson problem. $n$ and $N$ denote the number of neurons and subdomains, respectively. Note that the entries of weight $w_{j}^{s}$ are sampled from the interval $[-l,l]$.}
  \label{tab:initial}
  \centering
  \begin{tabular}{@{}c|ccc@{}}
         \toprule
         \multirow{2}{*}{Method} & \multicolumn{3}{c}{N} \\ \cmidrule(l){2-4}
          & $2 \times 2$ & $4 \times 4$ & $8 \times 8$ \\ \midrule
         He~$(l=\sqrt{2/n_N})$ & 8.4230e-03 & 1.7173e-04 & 1.1454e-05 \\
         Manual~$(l=1)$ & 4.9389e-10 & 4.0850e-09 & 6.5790e-08 \\
         Suggested~$(l=\sqrt{n/64})$ & \textbf{2.8926e-10} & \textbf{7.5946e-10} & \textbf{2.9894e-09} \\ \bottomrule
  \end{tabular}
\end{table}

\cref{tab:initial} shows that our initialization provides better results than the commonly used He initialization and manual initialization.
Therefore, all experiments were conducted using the initialization method we introduced.

\subsection{Application to PDEs}
We now look at the performance of the proposed non-overlapping DDM for ELM.
All reported computation times for solving the model equations are averages of $10$ runs.
The finite element method~(FEM) solution on a $4096\times4096$ grid is employed in cases where the exact solution is unknown.
Note that the solution space is discretized by conforming P1 elements.
To solve the system derived from FEM, we used CGM with stop condition $10^{-9}$ on the relative norm of the residual.

\subsubsection{2-D Poisson equation}
We first consider the following 2-D Poisson problem:
\begin{equation}
  \label{eqn:poisson}
  \left\{\begin{aligned}
  -\Delta u &= f &&\text{ in } \Omega=(0,1)^{2}, \\
  u &= g &&\text{ on } \partial \Omega
  \end{aligned}\right.
\end{equation}
with the exact solution $u(x,y) = \sin (2\pi x) e^y$.
\cref{tab:poisson1} shows the relative $L^{2}$ and $H^{1}$ errors of the proposed DDM solution.
Note that the result of $N = 1 \times 1$ denotes the ELM solution trained on the whole domain $\Omega$.
The relative errors of the solutions generated by Algorithm~\ref{alg:iterative} are slightly larger than those of the ELM solution, but the wall-clock time is significantly smaller.

\begin{table}
  \caption{Relative $L^{2}$ and $H^{1}$ errors of the solutions of 2-D Poisson equation with the exact solution $u(x,y)=\sin(2\pi x)e^y$. $N$ denotes the number of subdomains and the solutions of the case of $N=1 \times 1$ is the ELM solution. For $N>1$, the solutions are computed by Algorithm~\ref{alg:iterative}.}
  \label{tab:poisson1}
  \centering
  \begin{tabular}{cccc}
         \toprule
         $N$ & $L^{2}$ rel. error & $H^{1}$ rel. error & Wall-clock time~(sec) \\
         \midrule \midrule
         $1 \times 1$ & 8.1550e-11 & 1.6095e-09 & 8617.13\\
         \midrule
         $2 \times 2$ & 1.0748e-10 & 1.0590e-09 & 315.83 \\
         $4 \times 4$ & 1.6192e-08 & 3.2441e-08 & 19.46 \\
         $8 \times 8$ & 2.6392e-08 & 8.1948e-08 & 3.48\\
         $16 \times 16$ & 1.0212e-07 & 1.5199e-07 & 10.43\\
         \bottomrule
  \end{tabular}
\end{table}
In addition, we test a more oscillatory case where $g\equiv 0$ and
\begin{equation}
  \label{eqn:grf}
\begin{gathered}
  f(\mathbf{x}) = \sum_{i=1}^{256} a_i \sin(w_i\cdot \mathbf{x} + b_i),\\
  a_i \sim \mathcal{N}(0, \alpha^4/256), \quad w_i \sim \mathcal{N}(0, \alpha^2I), \quad b_i \sim U((0, 2\pi)).
\end{gathered}
\end{equation}
In this formulation, $f$ approximates a Gaussian random field of varying oscillation depending on $\alpha$.
In this case, since the exact solution is unknown, the FEM solution is used as a reference.
A graphical depiction of the right hand sides and solutions for $\alpha=16$ and $32$ is shown in \cref{fig:poisson_fem}.
The results of Algorithm~\ref{alg:iterative} are summarized in Tables \ref{tab:poisson2} and \ref{tab:poisson3}.
Note that the wall-clock times decrease as the number of subdomains increases until the $16\times16$ case.
At this point the interface problem becomes large enough to overtake the cost of solving a least squares problem on a subdomain.

\begin{figure}
  \begin{subfigure}{0.495\textwidth}
    \centering
    \includegraphics[width=\linewidth]{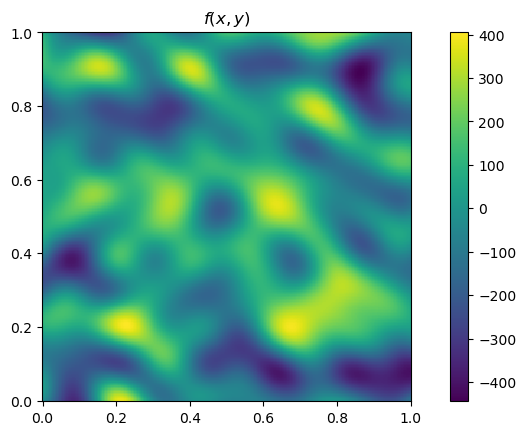}
    \caption*{(1a)}
  \end{subfigure}
  \hfill
  \begin{subfigure}{0.495\textwidth}
    \centering
    \includegraphics[width=.965\linewidth]{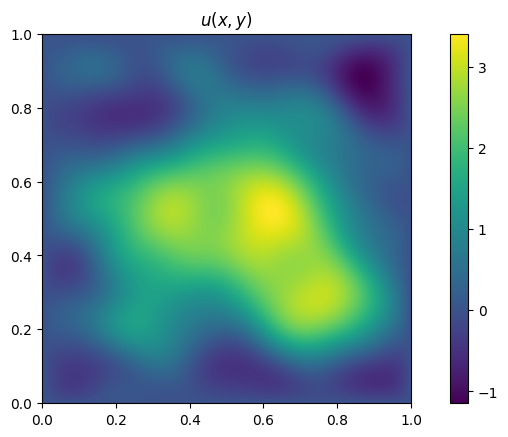}
    \caption*{(1b)}
  \end{subfigure}

  \begin{subfigure}{0.495\textwidth}
    \centering
    \includegraphics[width=\linewidth]{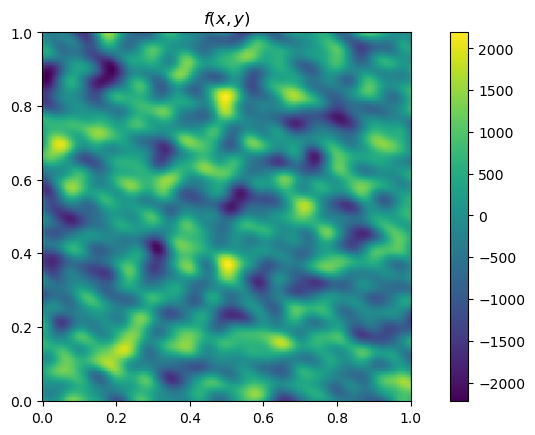}
    \caption*{(2a)}
  \end{subfigure}
  \hfill
  \begin{subfigure}{0.495\textwidth}
    \centering
    \includegraphics[width=.965\linewidth]{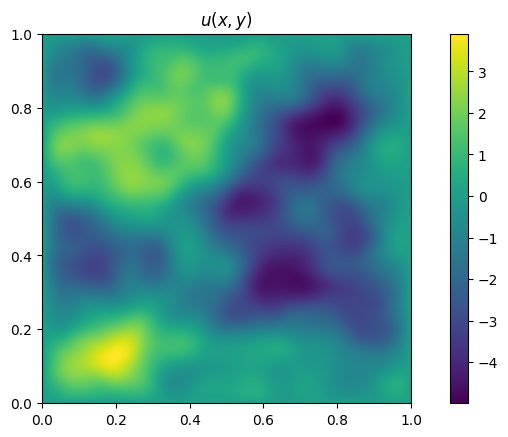}
    \caption*{(2b)}
  \end{subfigure}
  \caption{Graphical description of $f(x,y)$ and corresponding FEM solution $u(x,y)$ of 2-D Poisson equation; $u$ is computed via FEM on a $4{,}096 \times 4{,}096$ mesh.
  (1a) $f$ for $\alpha=16$ (1b) Solution $u$ for $\alpha=16$
  (2a) $f$ for $\alpha=32$ (2b) Solution $u$ for $\alpha=32$}
  \label{fig:poisson_fem}
\end{figure}
\begin{table}
  \caption{Relative $L^{2}$ and $H^{1}$ errors of the solutions of 2-D Poisson equation in the more oscillatory case when $\alpha=16$.
  The FEM solution on a $4{,}096 \times 4{,}096$ mesh is used as the reference solution.
  $N$ denotes the number of subdomains and the case of $N=1 \times 1$ is the ELM solution. For $N>1$, the solutions are computed by Algorithm~\ref{alg:iterative}.}
  \label{tab:poisson2}
  \centering
  \begin{tabular}{cccc}
         \toprule
         $N$ & $L^{2}$ rel. error & $H^{1}$ rel. error & Wall-clock time~(sec) \\
         \midrule \midrule
         $1 \times 1$ & 9.7629e-07 & 1.3357e-05 & 11003.49\\
         \midrule
         $2 \times 2$ & 7.5524e-07 & 1.0613e-05 & 313.48 \\
         $4 \times 4$ & 6.4662e-07 & 9.0596e-06 & 19.46 \\
         $8 \times 8$ & 6.8403e-07 & 9.5401e-06 & 3.36\\
         $16 \times 16$ & 5.7649e-07 & 8.4348e-06 & 8.32\\
         \bottomrule
  \end{tabular}
\end{table}
\begin{table}
  \caption{Relative $L^{2}$ and $H^{1}$ errors of the solutions of 2-D Poisson equation in the more oscillatory case when $\alpha=32$.
  The FEM solution on a $4{,}096 \times 4{,}096$ mesh is used as the reference solution.
  $N$ denotes the number of subdomains and the case of $N=1 \times 1$ is the ELM solution. For $N>1$, the solutions are computed by Algorithm~\ref{alg:iterative}.}
  \label{tab:poisson3}
  \centering
  \begin{tabular}{cccc}
         \toprule
         $N$ & $L^{2}$ rel. error & $H^{1}$ rel. error & Wall-clock time~(sec) \\
         \midrule \midrule
         $1 \times 1$ & 3.4758e-06 & 2.8890e-05 & 10992.84\\
         \midrule
         $2 \times 2$ & 2.7517e-06 & 2.2974e-05 & 314.09 \\
         $4 \times 4$ & 2.4976e-06 & 1.9726e-05 & 19.54 \\
         $8 \times 8$ & 2.4946e-06 & 2.1812e-05 & 3.28\\
         $16 \times 16$ & 2.2749e-06 & 1.9084e-05 & 8.11\\
         \bottomrule
  \end{tabular}
\end{table}

\subsubsection{2-D Poisson equation with a variable coefficient}
Next, we test the proposed DDM with an equation with variable coefficient
\begin{equation}
  \label{eqn:varpoisson}
  \left\{\begin{aligned}
  -\nabla \cdot (\rho \nabla u) &= f &&\text{ in } \Omega=(0,1)^2, \\
  u &= g &&\text{ on } \partial \Omega,
  \end{aligned}\right.
\end{equation}
where $\rho(\mathbf{x}) = \tanh(\mathrm{grf}(\mathbf{x})) + 1.1$, $f\equiv1$, and $g\equiv0$.
Here, $\mathrm{grf}$ is an approximate Gaussian random field defined in~\eqref{eqn:grf}.
We again test the cases $\alpha=16$ and $32$.
Note that $\rho$ varies from $0.1$ to $2.1$.
Since there is no closed form solution of~\eqref{eqn:varpoisson}, the FEM solution is used as the reference solution.
\cref{fig:varpoisson_fem} gives a graphical depiction of $\rho$ and the reference solution $u$.

\begin{figure}
  \begin{subfigure}{0.495\textwidth}
    \centering
    \includegraphics[width=\linewidth]{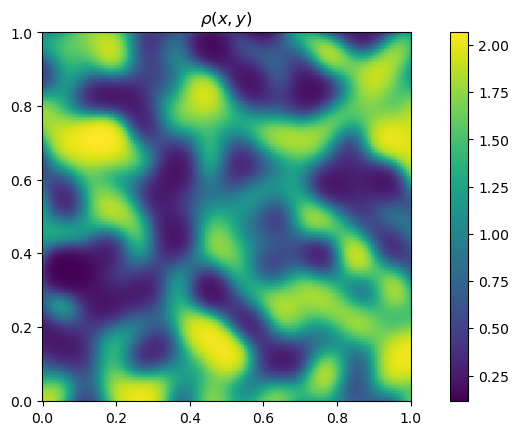}
    \caption*{(1a)}
  \end{subfigure}
  \hfill
  \begin{subfigure}{0.495\textwidth}
    \centering
    \includegraphics[width=.965\linewidth]{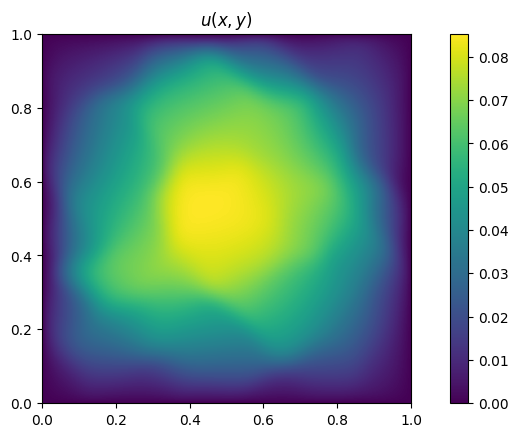}
    \caption*{(1b)}
  \end{subfigure}

  \begin{subfigure}{0.495\textwidth}
    \centering
    \includegraphics[width=\linewidth]{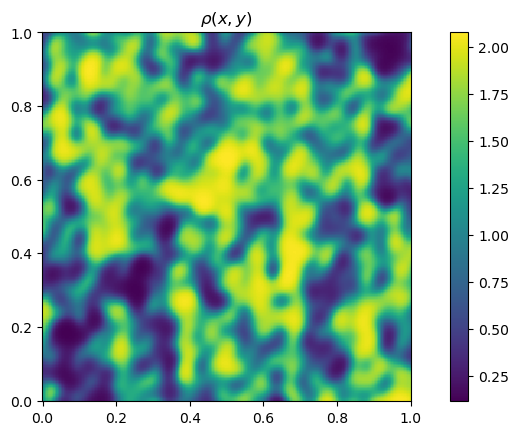}
    \caption*{(2a)}
  \end{subfigure}
  \hfill
  \begin{subfigure}{0.495\textwidth}
    \centering
    \includegraphics[width=.965\linewidth]{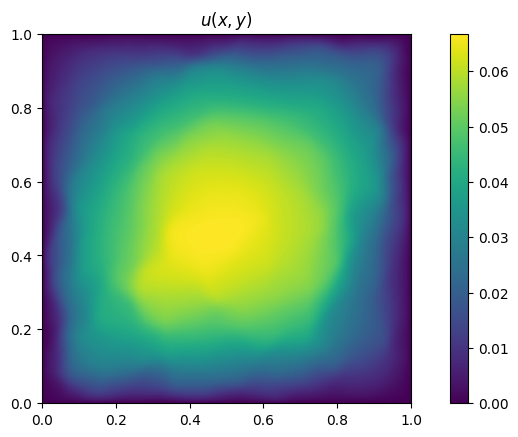}
    \caption*{(2b)}
  \end{subfigure}
  \caption{Graphical description of $\rho(x,y)$ and corresponding FEM solution $u(x,y)$ of 2-D Poisson equation with variable coefficient;
  $u$ is computed via FEM on a $4{,}096 \times 4{,}096$ mesh.
  (1a) $\rho$ for $\alpha=16$ (1b) Solution $u$ for $\alpha=16$
  (2a) $\rho$ for $\alpha=32$ (2b) Solution $u$ for $\alpha=32$}
  \label{fig:varpoisson_fem}
\end{figure}
Tables \ref{tab:varpoisson16} and \ref{tab:varpoisson32} show the relative $L^{2}$ and $H^{1}$ errors of the solutions of~\eqref{eqn:varpoisson} for $\alpha=16$ and $\alpha=32$, respectively.
In both cases, the best accuracy is achieved when $N=16\times16$, suggesting the proposed DDM is more effective in handling oscillatory problems.
In addition, the errors for $\alpha=16$ are much lower than that for $\alpha=32$.
This is in contrast to the FEM solution on a $256\times256$ mesh, which achieves relative $L^2$ errors of 2.4431e-04 and 3.7784e-04 for $\alpha=16$ and $32$, respectively.
This implies that ELMs fit low-frequency solutions much better, depending on the spectral bias property of the neural network~\cite{rahaman2019spectral}.

\begin{table}
  \caption{Relative $L^{2}$ and $H^{1}$ errors of the solutions of 2-D Poisson equation with variable coefficient when $\alpha=16$.
  The FEM solution on a $4{,}096 \times 4{,}096$ mesh is used as the reference solution.
  $N$ denotes the number of subdomains and the case of $N=1 \times 1$ is the ELM solution. For $N>1$, the solutions are computed by Algorithm~\ref{alg:iterative}.}
  \label{tab:varpoisson16}
  \centering
  \begin{tabular}{cccc}
         \toprule
         $N$ & $L^{2}$ rel. error & $H^{1}$ rel. error & Wall-clock time~(sec) \\
         \midrule \midrule
         $1 \times 1$ & 1.8452e-06 & 4.3370e-05 & 10990.45 \\
         \midrule
         $2 \times 2$ & 3.3742e-06 & 5.7660e-05 & 318.36 \\
         $4 \times 4$ & 6.9140e-06 & 4.1558e-05 & 20.67\\
         $8 \times 8$ & 3.3147e-06 & 2.1626e-05 & 4.54 \\
         $16 \times 16$ & 7.3209e-07 & 5.5242e-06 & 15.35 \\
         \bottomrule
  \end{tabular}
\end{table}
\begin{table}
  \caption{Relative $L^{2}$ and $H^{1}$ errors of the solutions of 2-D Poisson equation with variable coefficient when $\alpha=32$.
  The FEM solution on a $4{,}096 \times 4{,}096$ mesh is used as the reference solution.
  $N$ denotes the number of subdomains and the case of $N=1 \times 1$ is the ELM solution. For $N>1$, the solutions are computed by Algorithm~\ref{alg:iterative}.}
  \label{tab:varpoisson32}
  \centering
  \begin{tabular}{cccc}
         \toprule
         $N$ & $L^{2}$ rel. error & $H^{1}$ rel. error & Wall-clock time~(sec) \\
         \midrule \midrule
         $1 \times 1$ & 7.6186e-03 & 3.3764e-02 & 11020.72 \\
         \midrule
         $2 \times 2$ & 7.4366e-03 & 5.7018e-02 & 315.94 \\
         $4 \times 4$ & 5.8671e-03 & 5.6910e-02 & 20.61 \\
         $8 \times 8$ & 4.6123e-03 & 2.2875e-02 & 4.56 \\
         $16 \times 16$ & 6.0091e-04 & 4.8583e-03 & 14.10\\
         \bottomrule
  \end{tabular}
\end{table}

\subsubsection{Plate Bending}
Finally, we test a plate bending problem,
assuming the Kirchhoff-Love plate bending theory for isotropic and homogeneous plates.
For Young's modulus $E$, plate thickness $H$, and Poisson's ratio $\nu$, the flexural rigidity is given by $D=EH^3/12(1-\nu^2)$.
The equation for the simply supported plate is given by
\begin{equation*}
  \left\{\begin{aligned}
  \Delta^2 u &= \frac{q}{D} &&\text{ in } \Omega=(0,1)^2, \\
  u &= 0 &&\text{ on } \partial \Omega,\\
  \Delta u &= 0 &&\text{ on } \partial \Omega,
  \end{aligned}\right.
\end{equation*}
where $q$ is the transversal load.
The transmission conditions are continuity in $u$ and $\Delta u$ and their normal derivative across the interface.
We solve the equation with $q=\sin(\pi x)\sin(\pi y)$, $E=10^7$, $H=10^{-3}$, and $\nu=0.3$, and the exact solution is $u=\sin(\pi x)\sin(\pi y) / 4\pi^4D$.
\cref{tab:platebending} presents the relative $L^{2}$ and $H^{1}$ errors of the results of Algorithm~\ref{alg:iterative}.

\begin{table}
  \caption{Relative $L^{2}$ and $H^{1}$ errors of the solutions of a 2D plate bending problem.
  $N$ denotes the number of subdomains and the case of $N=1 \times 1$ is the ELM solution.
  For $N>1$, the solutions are computed by Algorithm~\ref{alg:iterative}.}
  \label{tab:platebending}
  \centering
  \begin{tabular}{cccc}
         \toprule
         $N$ & $L^{2}$ rel. error & $H^{1}$ rel. error & Wall-clock time~(sec) \\
         \midrule \midrule
         $1 \times 1$ & 8.7574e-09 & 1.2857e-07 & 11105.85 \\
         \midrule
         $2 \times 2$ & 1.1559e-08 & 1.6362e-07 & 322.29 \\
         $4 \times 4$ & 6.9948e-09 & 1.1603e-07 & 21.93 \\
         $8 \times 8$ & 3.6260e-08 & 7.3627e-07 & 3.74  \\
         $16 \times 16$ & 4.7774e-07 & 5.5870e-06 & 5.55 \\
         \bottomrule
  \end{tabular}
\end{table}

\section{Conclusion}
\label{Sec:Conc}
In this paper, we proposed a novel nonoverlapping domain decomposition method for ELM that is suitable for distributed memory computing.
Motivated by substructuring methods in numerical analysis, a local neural network was defined in each nonoverlapping subdomain and an auxiliary variable $u_{\Gamma}$ at the interface was introduced.
The coefficients of the last layer of local neural networks were eliminated using Schur complements, which gave a reduced system for $u_{\Gamma}$.
After that, $u_{\Gamma}$ was obtained by solving the system and the coefficients of the last layer of local neural networks were obtained using the auxiliary variable $u_{\Gamma}$ in parallel.
We also presented an initialization technique when the local neural network is equipped with a shallow neural network, and observed that our technique shows better performance than the commonly used He initialization.
Numerical experiments demonstrated the practical efficacy of the proposed nonoverlapping method when large number of neurons are used.
We expect that the proposed nonoverlapping DDM can be efficiently utilized for training ELMs with a large number of neurons through distributed parallel computation.
As a future direction, it would be interesting to achieve scalability by either introducing a coarse problem or developing a preconditioner.

\bibliographystyle{siam}
\bibliography{ddelm}

\end{document}